\documentclass{amsart}

\usepackage{amsthm, mathrsfs,amssymb,amsmath}
\usepackage{enumerate}

\makeatletter
\@namedef{subjclassname@2010}{%
\textup{2010} Mathematics Subject Classification}
\makeatother

\allowdisplaybreaks

\newtheorem{thm}{Theorem}[section]

\newtheorem{cor}[thm]{Corollary}

\theoremstyle{definition}

\newtheorem{exmp}{Example}[section]

\theoremstyle{remark}

\newcommand{\spt}{\textnormal{spt}}

\title[ Topological transitive  sequence of cosine operators on Orlicz space]
{ Topological transitive sequence of cosine operators on Orlicz
space}
\author{\bf  I. Akbarbaglu, M. R. Azimi and V. Kumar}

\address{Ibrahim Akbarbaglu \endgraf Department of mathematics
\endgraf farhangian university
\endgraf tabriz, iran.}
\email{ibrahim.akbarbaglu@gmail.com}

\address {Mohammad Reza Azimi \endgraf Department of mathematics
\endgraf faculty of sciences
\endgraf university of maragheh
\endgraf 55181-83111, maragheh, iran.}
\email{mhr.azimi@maragheh.ac.ir}

\address{Vishvesh Kumar \endgraf Department of Mathematics
\endgraf Indian Institute of Technology Delhi \endgraf Delhi - 110 016, India.}
\email{vishveshmishra@gmail.com}

\begin{document}

\begin{abstract}
For a Young function $\phi$ and a locally compact second countable
group $G,$ let $L^\phi(G)$ denote the Orlicz space on $G.$ In this
article, we present
 a necessary and sufficient condition for the topological transitivity
  of a sequence of cosine operators
  $\{C_n\}_{n=1}^{\infty}:=\{\frac{1}{2}(T^n_{g,w}+S^n_{g,w})\}_{n=1}^{\infty}$,
   defined on $L^{\phi}(G)$. We investigate the
   conditions for a sequence of  cosine operators to be topological mixing. Moreover, we go
    on to prove the similar results for the direct sum of a sequence
     of the cosine operators. At the last, an example of topological
     transitive sequence of  cosine operators is given.

\end{abstract}

\keywords{ Hypercyclicity; Topological transitive; Topological mixing; Weighted translation
 operator; Orlicz space; Locally compact groups }
\subjclass[2010]{Primary 47A16, 46E30; Secondary 22D05}
\maketitle

\section{Introduction and preliminaries}
A sequence of bounded linear  operators $\{T_n\}_{n=1}^{\infty}$
acting on a Fr\'{e}chet space $X$ is said to be \emph{topological
transitive} if for each pair of no-empty open sets $(U, V)$ in $X$,
there exists an $n\in \mathbb{N}$ such that $T_n(U)\cap V\neq
\emptyset$. A single bounded linear operator $T$ is topological
transitive whenever the sequence of its iterates, that is
$\{T^n\}_{n=0}^{\infty}$ is topological transitive where $T^0$ is
the identity map.
 Furthermore, a sequence $\{T_n\}_{n=1}^{\infty}$
   is called \emph{hypercyclic} if there is a vector $x \in X$
whose orbit $\{T_n x : n=0, 1, 2,... \}$ is dense in $X$. Such a
vector is called a hypercyclic vector for that sequence.
Analogously, when the sequence $\{T^n\}_{n=0}^{\infty}$ is dense in
$X$, we say that an operator $T$ is hypercyclic.  It is worth noting
that these two notions, topological transitivity and hypercyclicity
in a single case are more likely equivalent on a Fr\'{e}chet space
$X$ \cite{bay,mang}. An operator $T$ is \emph{topologically mixing}
whenever for each pair of no-empty open sets $(U, V)$ in $X$, there
exists an $N\in \mathbb{N}$ such that $T^n(U)\cap V\neq \emptyset$
for all $n\geq N$. The operators of the form "identity plus a
backward shift" are the example of topologically mixing operators
which are also hypercyclic.
 An operator $T$  on a
Fr\'{e}chet space $X$ is \emph{weakly mixing } if and only if
$T\oplus T$  is hypercyclic on $X\oplus X$. Note that weakly mixing
maps are topologically transitive but in the topological setting,
the converse is not true. More detailed information concerning
dynamics of linear operators  may be found in the best interesting
books \cite{bay} and \cite{mang}.
\\  A continuous, even and  convex
function $\phi: \mathbb{R}\rightarrow \mathbb{R}^+\cup \{0\}$ is
called a\emph{ Young's function} whenever $\phi(0)=0$ and
$\lim_{t\rightarrow \infty} \phi(t)=\infty$.
  Corresponding for each Young's function $\phi$, there is  another Young's function
  $\psi: \mathbb{R}\rightarrow \mathbb{R}^+\cup\{0\}$ defined by
$\psi(y):=\sup \{x|y|-\phi(x): \quad x\geq 0 \}$,
  which is called \emph{complementary Young's function} of
   $\phi$.\\
 Let $G$
be a locally compact and second countable group with the identity
element $e$. Consider a right invariant Haar measure $\lambda$ on
$G$. Let $L^{\phi}(G)$ denote the set of all Borel measurable
functions $f$ on $G$ such that
$$\int_G |f(x)\nu(x)|\, d\lambda(x)<+\infty,$$ for each measurable
$\nu$ with $$\Psi(\nu):= \int_G \psi(\nu)d\lambda \leq 1.$$  The set
of all Borel functions $\nu$ on $G$ such that $\Psi(\nu) \leq 1$
will be denoted by $\Omega$.  Plainly $L^{\phi}(G)$ is a vector
space. Now we assume that the Young's function $\phi$ vanishes only
at zero. This guarantees that $L^{\phi}(G)$  equipped with the norm
$$\|f\|_\phi:= \sup_{\Psi(\nu)\leq 1} \int_G |f(x)\nu(x)|\, d\lambda(x),$$
 is a Banach space and called an \emph{Orlicz space}\cite{rao}.\\
Another equivalent norm on $L^{\phi}(G)$ is defined by
$$N_\phi(f)=\inf \{k>0: \int_G \phi(\frac{|f|}{k})d\mu\leq 1\},$$
which is so-called the \emph{Luxemburg} norm.

A Young's function $\phi$ is said to satisfy condition
\emph{$\Delta_2$-regular} if there is a constant $k>0$ such that
$\phi(2t)\leq k\phi(t)$ for large values of $t$ when
$\lambda(G)<\infty$. In case $\lambda(G)=\infty$, $\phi(2t)\leq
k\phi(t)$ for each $t>0$. Some examples of such Young's functions
may be found in \cite[Example 2.8]{az}. If $\phi$ is
$\Delta_2$-regular, then the space $C_c(G)$ of all continuous
functions on $G$ with compact support is dense in $L^{\phi}(G)$, and
the dual space $(L^{\phi}(G), \|\cdot\|_{\phi})$ is $(L^{\psi}(G),
N_{\psi}(\cdot))$. For further information the interested reader is
referred to \cite{rao}.

 It is well known that the hypercyclic
phenomenon is occurred  only on infinite-dimensional and separable
spaces \cite{bay, mang}. For this reason, we assume that $G$ is
second countable and $\phi(x)=0$ if and only if $x=0$ \cite[p. 87,
Theorem1]{rao}. Throughout this paper, the Banach space of all
essentially bounded and measurable functions on $G$ is denoted by
$L^\infty(G)$ and $N(f, r)$ denotes a neighborhood of $f\in
L^{\phi}(G)$ with radius $r>0$. A bounded continuous function $w:
G\rightarrow (0, \infty)$ is called a \emph{weight}. For $g\in G$,
let $\delta_g$ be the unit point mass at $g$. Given a weight $w$ on
$G$ and $g\in G$, a \emph{weighted translation} $T_{g,w}:
L^{\phi}(G)\rightarrow L^{\phi}(G)$ is defined by
$$T_{g,w}(f):= w\cdot f\ast \delta_g,  \quad    \quad f\in L^{\phi}(G)$$
where $f\ast \delta_g$ is the following convolution
$$f\ast \delta_g(t):=\int_G f(tx^{-1})d\delta_g(x)=f(tg^{-1}), \quad \quad  t\in G.$$
Indeed it is the right translation of $f$ by $g^{-1}$. Moreover, it
is easy to check that $f\ast \delta_g \in L^{\phi}(G)$ whenever
$f\in L^{\phi}(G)$. Recall that an element $g\in G$ is called a
\emph{torsion element} if it is of finite order. An element $g\in G$
is called \emph{periodic}  if the closed subgroup $G(g)$ generated
by $g$ is compact. Further, an element in $G$ is \emph{aperiodic} if
it is not periodic. Equivalently,  $g\in G$ is an aperiodic element,
if and only if for any compact subset $K\subset G$, there exists an
$N\in \mathbb{N}$ such that $K\cap Kg^{-n}=\emptyset$ for $n> N$
\cite[Lemma 2.1]{chen}. It is worth noting that a weighted
translation $T_{g,w}$ cannot be hypercyclic whenever
$\|w\|_{\infty}\leq 1$ or $g$ is a torsion element \cite{chen, az}.
 The hypercyclic weighted translation on locally
compact groups have been characterized by C. Chen  \cite{chen} in
details. In addition, he has studied the hypercyclicity of weighted
convolution operators in \cite{chen2}. Moreover, the hypercyclic
weighted translations on Orlicz spaces $L^{\phi}(G)$ has been
studied in \cite{az}.
 In the case $w^{-1}:=\frac{1}{w}\in L^{\infty}(G)$, the weighted
translation operator $T_{g,w}$ is invertible and its inverse is
$T_{g^{-1}, w^{-1}\ast \delta_g}$ which will be denoted by $S_{g,w}$
throughout this paper. For each $n\in \mathbb{Z}$, the \emph{cosine
operator} $C_n: L^{\phi}(G)\rightarrow L^{\phi}(G)$ is defined by
$$C_n:=\frac{1}{2}(T^n_{g,w}+S^n_{g,w}).$$
The study of cosine operator on Banach spaces is originally due to
the work done in \cite{bon} by Bonilla and Miana. They gave
sufficient conditions for the hypercyclicity and topological mixing
of a strongly continuous cosine operator function. Afterwards, T.
Kalmes in \cite{kal} characterized the hypercyclicity of cosine
operator functions on $L^p(\Omega)$ ($\Omega$ is open subset of
$\mathbb{R}^d$) generated by second order partial differential
operators. He also showed that the hypercyclicity and weakly mixing
of these type of operators are equivalent.

Furthermore, a necessary and sufficient condition for the
topological transitivity of the cosine operator $C_n$ on $L^p(G)$
has been already studied in \cite{chen3}. Nevertheless, in this
paper, we are going to generalize  that condition to Orlicz space
$L^{\phi}(G)$ on which the topological transitivity of the cosine
operator $C_n$ stays still in force.


\section{Main Result}
In this section we present our main results with some immediate consequences.
 We begin with the following theorem which give a necessary and sufficient condition
  on weight so that cosine operator is topological transitive.

\begin{thm}\label{T1} Let $g \in G$ be an aperiodic element of $G$ and let $\phi$
 be a $\Delta_2$-regular Young's function. Let $w, \,w^{-1} \in L^\infty(G).$
  If $C_n:= \frac{1}{2}(T^n_{g,w}+S^n_{g,w})$ is a cosine operator on
  $L^\phi(G)$,
  then the following statements are equivalent.
\begin{itemize}
\item[(i)] $(C_n)_{n \in \mathbb{N}_0}$ is topological transitive.
\item[(ii)] For each non-empty compact subset $K \subset G$ with $\lambda(K)>0,$ there exist
 sequences of Borel sets $(E_k),\, (E_k^+)$ and $(E_k^-)$ in $K,$ and a sequence
 $(n_k)$ of positive numbers such that for $E_k=E_k^+ \cup E_k^-,$ we have
$$\lim_{k \rightarrow \infty} \sup_{\nu \in \Omega} \int_{K \backslash E_k} |\nu(x)|\, d\lambda(x)=0.$$
 Moreover, the two sequence
$$\varphi_n= \prod_{j=1}^n w*\delta_{g^{-1}}^j\,\,\,\,\,\text{and}\,\,\,\, \tilde{\varphi_n}=
 \left( \prod_{j=1}^n w*\delta_g^j \right)^{-1}$$ satisfy
$$ \lim_{k \rightarrow \infty} \sup_{\nu \in \Omega} \int_{E_k}
 \varphi_{n_k}(x)|\nu(xg^{n_k})|\, d\lambda(x)=0, $$
$$ \lim_{k \rightarrow \infty} \sup_{\nu \in \Omega} \int_{E_k}
 \tilde{\varphi}_{n_k}(x)|\nu(xg^{n_k})|\, d\lambda(x)=0,$$
$$\lim_{k \rightarrow \infty} \sup_{\nu \in \Omega} \int_{E_k^+}
 \varphi_{2n_k}(x)|\nu(xg^{2n_k})|\, d\lambda(x)=0,$$
$$ \lim_{k \rightarrow \infty} \sup_{\nu \in \Omega} \int_{E_k^-}
 \tilde{\varphi}_{2n_k}(x)|\nu(xg^{2n_k})|\, d\lambda(x)=0.$$
\end{itemize}
\end{thm}
\begin{proof}
(i) $\Rightarrow$ (ii). In spite of being different underlying
spaces, the approach of the proof is followed like as done in
\cite{chen3}. Let $K$ be a compact subset of $G$ such that
$\lambda(K)>0.$ Since $g \in G$
 is an aperiodic element, there exists $N \in \mathbb{N}$ such that
  $K \cap K g^{\pm n} = \emptyset$ for $n > N$ \cite[Lemma 2.1]{chen}.
  Denote the characteristic function of $K$ defined on $G$ by $\chi_K.$
  Clearly $\chi_K\in L^\phi(G)$. Take $U=N(\chi_K, \epsilon^2) $ and $V=N(-\chi_K,
  \epsilon^2)$ in the definition of the topological transitive for
  the sequence $(C_n)_{n\in \mathbb{N}_0}$.
  Then for each $\epsilon \in (0,1),$ there exist $f \in L^\phi(G)$ and $m \in N$,  such that
$$\| f- \chi_K\|_\phi < \epsilon^2\,\,\,\,\,\text{and}\,\,\,\, \|C_mf+\chi_K\|_\phi< \epsilon^2.$$
Hence, we can write that $$\| Re(f)- \chi_K\|_\phi <
\epsilon^2\,\,\,\,\,\text{and}\,\,\,\,
\|Re(C_mf)+\chi_K\|_\phi=\|C_m Re(f)+\chi_K\|_\phi< \epsilon^2,$$
where $Re(f)$ is the real part of the complex valued function $f$.
 Since the maps $Re: L^\phi(G, \mathbb{C}) \rightarrow
 L^\phi(G, \mathbb{R})$ and $f \mapsto f^+:=\text{max}\{0,f\}$ from $L^\phi(G, \mathbb{R})$
   to $L^\phi(G, \mathbb{R}) $ are continuous and also commute with both $T_{g,w}$ and
   $S_{g,w}$, hence without loss of generality  we may assume that the
   function $f$ is positive.

Therefore, for any Borel subset $F \subset G,$ we have
\begin{eqnarray} \label{1}
\|C_m f^+\chi_F\|_\phi &\leq& \|(C_m f)^+\|_\phi = \|(C_m
f+\chi_K-\chi_K)^+\|_\phi \nonumber
 \\&\leq & \|(C_mf+\chi_K)^+\|_\phi+\|(-\chi_K)^+\|_\phi\nonumber \\
  &=& \|(C_mf+\chi_K)^+\|_\phi \leq \|C_mf+\chi_K\|_\phi < \epsilon^2.
\end{eqnarray}
Set $A=\{x \in K:|f(x)-1|> \epsilon\}.$  Then
\begin{eqnarray*}
\epsilon^2 > \|f-\chi_K\|_\phi &=& \sup_{\nu \in \Omega} \int_G |f(x)-\chi_K(x)| |\nu(x)|\, d\lambda(x) \\
&\geq & \sup_{\nu \in \Omega} \int_K |f(x)-1| |\nu(x)|\, d\lambda(x) \\
&>& \sup_{\nu \in \Omega} \int_A \epsilon |\nu(x)|\,d\lambda(x).
\end{eqnarray*}
Therefore, we have  $$\sup_{\nu \in \Omega} \int_A |\nu(x)| \,
d\lambda(x)< \epsilon .$$ Set $B_m=\{x \in K: |C_mf(x)+1| \geq
\epsilon\}.$ Then, by the similar argument,  we get $$\sup_{\nu \in
\Omega} \int_{B_m} |\nu(x)|\, d \lambda(x)<\epsilon.$$ Now, let
$E_m:=\{x \in K : |f(x)-1|<\epsilon\} \cap \{x\in K:
|C_mf(x)+1|<\epsilon\}.$ Then, for $x \in E_m,$ we get
$f(x)>1-\epsilon>0$ and $C_m f(x)< \epsilon-1<0.$ Also,
\begin{eqnarray*}
\sup_{\nu \in \Omega} \int_{K\setminus E_m} |\nu(x)|\, d\lambda(x)
& =& \sup_{\nu \in \Omega} \int_{A \cup B_m} |\nu(x)|\, d\lambda(x) \\
&=& \sup_{\nu \in \Omega} \int_{A} |\nu(x)|\, d\lambda(x)+ \sup_{\nu
\in \Omega} \int_{B_m} |\nu(x)|\, d\lambda(x)\\
&<&\epsilon+\epsilon=2\epsilon.
\end{eqnarray*}
By keeping the facts that Haar measure $\lambda$ is right
invariant, $T_{g,w}^mf^+$ and $S_{g,w}^mf^+$ are positive in the
mind, with the aid of \eqref{1} we get that
 \begin{eqnarray*}
2\epsilon^2 &>&\|2(C_mf^+)\chi_{E_m}\|_\phi =
\|(T_{g,w}^mf^++S_{g,w}^mf^+)\chi_{E_m}\|_\phi \geq \|T_{g,w}^mf^+\chi_{E_m}\|_\phi \\
 &=& \sup_{\nu \in \Omega} \int_{E_m} |T_{g,w}^mf^+(x)| |v(x)| \, d\lambda(x) \\
  &=& \sup_{\nu \in \Omega} \int_{E_m} |w(x)w(xg^{-1}) \ldots w(xg^{-m+1})f^+(xg^{-m})|
   |v(x)| \, d\lambda(x) \\ &=& \sup_{\nu \in \Omega} \int_{E_m} w(xg^m)w(xg^{m-1})
    \ldots w(x g)f^+(x) |v(xg^m)| \, d\lambda(x)\\
    &=& \sup_{\nu \in \Omega} \int_{E_m} \varphi_m(x) f^+(x) |\nu(xg^m)|\, d\lambda(x) \\
    &>& \sup_{\nu \in \Omega} \int_{E_m} (1-\epsilon)\varphi_m(x) |\nu(xg^m)|\,d\lambda(x).
\end{eqnarray*}
Therefore,
$$\sup_{\nu \in \Omega} \int_{E_m} \varphi_m(x) |\nu (xg^m)|\,d\lambda(x) < \frac{2\epsilon^2}{1-\epsilon}.$$
By the similar argument, we get
$$2 \epsilon^2 >\|(S_{g,w}^mf^+)\chi_{E_mg^m}\|_\phi> (1-\epsilon) \sup_{\nu \in \Omega} \int_{E_m} \tilde{\varphi_m}(x)|v(xg^m)|\,d\lambda(x)$$ and thus
$$\sup_{nu \in \Omega} \int_{E_m} \tilde{\varphi}_m(x) |\nu (xg^m)|\,d\lambda(x) < \frac{2\epsilon^2}{1-\epsilon}.$$ Hence,
the first part of Condition (ii) holds as $\epsilon$ is arbitrary.

Now, let $E_m^-= \{x\in E_m: T_{g,w}^mf(x)< \epsilon-1\}$ and
$E_m^+= E_m \setminus E_m^-.$ Then, for $x \in E^+_m,$ we have
$$\epsilon-1>C_mf(x)= \frac{1}{2}T_{g,w}^mf(x)+\frac{1}{2} S_{g,w}^mf(x)
 \geq \frac{1}{2}(\epsilon-1)+\frac{1}{2}S_{g,w}^mf(x)$$ and therefore
  $$S_{g,w}^mf(x)<\epsilon-1,\,\,\,\,\,\,x \in E_m^+.$$
Now, consider the following
\begin{eqnarray*}
&&(1-\epsilon) \sup_{\nu \in \Omega} \int_{E_m^+} \varphi_{2m}(x)
|\nu(x g^{2m})|\, d\lambda(x)\\
&=&\sup_{\nu \in \Omega} \int_{E_m^+} |w(xg^{2m})
w(xg^{2m-1})w(xg^{2m-2}) \ldots w(x g)|\, |S_{g,w}^mf^-(x)|\,
|\nu(xg^{2m})|\, d\lambda(x)\\
&\leq&  \sup_{\nu \in \Omega} \int_{E_m^+g^{2m}} |w(x) w(x g)w(xg^2)
\ldots w(xg^{-(2m-1)})| \,|S_{g,w}^mf^-(xg^{-2m})|\, |\nu(x)|\,
d\lambda(x)\\
&=& \sup_{\nu \in \Omega} \int_{E_m^+g^{2m}} |T_{g,w}^{2m}
S_{g,w}^mf^-(x)|\,|\nu(x)|\,d\lambda(x)\\
&=& \sup_{\nu \in \Omega}
\int_{E_m^+g^{2m}} |T_{g,w}^{m}f^-(x)|\,|\nu(x)|\,d\lambda(x)\\
&\leq& 2 \sup_{\nu \in \Omega} \int_{E_m^+ g^{2m}}
|C_mf^-(x)|\,|\nu(x)|\, d\lambda(x)\\
&=& 2 \sup_{\nu \in \Omega}
\int_{G} |C_mf^-(x)\chi_{E_m^+g^{2m}}|\,|\nu(x)|\, d\lambda(x)\\
&=&2 \sup_{\nu \in \Omega} \int_{G}
|C_m(f^+-f)(x)\chi_{E_m^+g^{2m}}|\,|\nu(x)|\, d\lambda(x)\\
&=&2 \sup_{\nu \in \Omega} \int_{G}
|(C_mf^+)\chi_{E_m^+g^{2m}}-(C_mf+\chi_K) \chi_{E_m^+
g^{2m}}+\chi_{K \cap E_m^+g^2m}|\, |\nu(x)|\, d\lambda(x)\\
&\leq& 2\sup_{\nu \in \Omega} \int_{G} |(C_mf^+)\chi_{E_m^+g^{2m}}|
|\nu(x)|\,d\lambda(x)\\
&+& 2\sup_{\nu \in \Omega} \int_{G} |(C_mf+\chi_K) \chi_{E_m^+
g^{2m}}|\,|\nu(x)|\, d\lambda(x)+ 2\sup_{\nu \in \Omega} \int_{K
\cap E_m^+g^{2m}} |\nu(x)|\,d\lambda(x)\\
&\leq&  2 \|(C_mf^+ \chi_{E_m^+g^{2m}})\|_\phi+2
\|(C_mf+\chi_K)\|_\phi+2 \|\chi_{K \cap E_m^+g^{2m}}\|_\phi\\
&<& 2 \epsilon^2+2\epsilon^2+0=4\epsilon^2.
\end{eqnarray*}
In the last inequality, from the fact $K\cap Kg^{\pm 2m}=\emptyset$,
has been already used. Therefore, we get $$ \sup_{\nu \in \Omega}
\int_{E_m^+} \varphi_{2m}(x) |\nu(x g^{2m})|\, d\lambda(x)
 < \frac{4\epsilon^2}{(1-\epsilon)}.$$ In similar lines, we also have
$$\sup_{\nu \in \Omega} \int_{E_m^-} \tilde{\varphi}_{2m}(x) |\nu(x g^{2m})|\, d\lambda(x)
 < \frac{4\epsilon^2}{(1-\epsilon)}.$$
Since $\epsilon$ is arbitrary, last two condition of (ii) part also fulfilled.

(ii) $\Rightarrow$ (i). Let $U$ and $V$ be two non-empty open subsets of $L^\phi(G).$
 Since $\phi$ is $\Delta_2$-regular we can choose two non-zero functions $f$ and $h$ in
 $C_c(G)$ such that $f \in U$ and $h \in V.$ Set $K=supp(f) \cup
 supp(h)$, the supports of $f$ and $h$ respectively.
  Let $E_k \subset K$ and it satisfies condition (ii).
  But $g \in G$ is an aperiodic element, hence there exists
   $M \in \mathbb{N}$ such that $K \cap K g^{\pm n} = \emptyset$ for all $n >M.$
  Subsequently, for a given $\epsilon>0$, one can find $N \in \mathbb{N}$
   such that for each $k>N$,  $n_k>M$ and
$$\|h\|_\infty \cdot \sup_{\nu \in \Omega} \int_{E_k}
 \varphi_{n_k}(x)|\nu(xg^{n_k})|\, d\lambda(x)< \epsilon,
  \,\,\,\,\,\,\,\|h\|_\infty \sup_{nu \in \Omega} \int_{K \backslash E_k} |\nu(x)|
  d \lambda(x)< \epsilon.$$
Now, we have
\begin{eqnarray*}
\|T_{g,w}^{n_k}(h \chi_{E_k})\|_\phi & =& \sup_{\nu \in \Omega} \int_G |T_{g,w}^{n_k}
(h \chi_{E_k})(x)\nu(x)|\, d\lambda(x) \\
 &=&  \sup_{\nu \in \Omega} \int_G |w(x)w(xg^{-1})
\ldots w(xg^{-n_k+1})h(xg^{-n_k}) \chi_{E_k}(xg^{-n_k})\nu(x)| \, d\lambda(x) \\
&=&\sup_{\nu \in \Omega} \int_G  |w(xg^{n_k})w(xg^{n_k-1}) \ldots
w(x
g)h(x)\chi_{E_k}(x) \nu(x)| \, d\lambda(x) \\
&\leq & \|h\|_\infty\cdot \sup_{\nu \in \Omega} \int_{E_k}
\varphi_{n_k}(x) |\nu(xg^{n_k})|\,d\lambda(x)<\epsilon.
\end{eqnarray*}
Hence, $$\lim_{k \rightarrow \infty} \|T_{g,w}^{n_k}(h
\chi_{E_k})\|_\phi =0.$$ Also,
\begin{eqnarray*}
\|h-h\chi_{E_k}\|_\phi &=& \sup_{\nu \in \Omega} \int_G |h(x)-h(x)\chi_{E_k}(x)||\nu(x)|\,d\lambda(x)\\
 &=& \sup_{\nu \in \Omega} \int_G |h(x)\chi_{K \backslash E_k}(x)||\nu(x)|\, d\lambda(x)\\
 &=& \int_{K \backslash E_k} |h(X)| |\nu(x)|\,d\lambda(x)\\
&\leq & \|h\|_\infty \cdot \int_{K \backslash E_k} |\nu(x)|\,
d\lambda(x) < \epsilon.
\end{eqnarray*}
By similar argument, using the conditions given in (ii) we get
$$\lim_{k \rightarrow \infty} \|S_{g,w}^{n_k}(h \chi_{E_k})\|_\phi  =0,\,\,
\lim_{k \rightarrow \infty} \|S_{g,w}^{2n_k}(h \chi_{E_k^-})\|_\phi=
 \lim_{k \rightarrow \infty}\|T_{g,w}^{2n_k}(h\chi_{E_k^+})\|_\phi=0.$$
  In addition, we also have $$\lim_{k \rightarrow \infty}\|S_{g,w}^{n_k}(f \chi_{E_k})\|_\phi=
   \lim_{k \rightarrow \infty} \|T_{g,w}^{n_k}(f\chi_{E_k})\|_\phi=0.$$
For each $k \in \mathbb{N},$ we set
$$v_k=f \chi_{E_k}+2T_{g,w}^{n_k}(h\chi_{E_k^+})+2S_{g,w}^{n_k}(h\chi_{E_k^-}).$$
In this stage, an application of the frequently used fact i.e., $K
\cap K g^{(m_1-m_2)n_k}= \emptyset$ ($m_1,m_2 \in \mathbb{Z}$ and
$m_1 \neq m_2$) with Minkowski inequality yield that
$$ \|v_k-f\|_\phi \leq  \|f-f\chi_{E_k}\|_\phi+
2 \|T_{g,w}^{n_k}(h\chi_{E_k^+})\|_\phi+ 2\|S_{g,w}^{n_k}(h
\chi_{E_k^-})\|_\phi $$ and
$$\|C_{n_k}v_k-h\|_\phi \leq \|h-h\chi_{E_k}\|_\phi+
 \frac{1}{2} \|T_{g,w}^{n_k}(f\chi_{E_k})\|_\phi+\frac{1}{2} \|S_{g,w}^{n_k}(f\chi_{E_k})\|_\phi+
  \|T_{g,w}^{2n_k}(h \chi_{E_k^+})\|_\phi+ \|S_{g,w}^{n_k}(h\chi_{E_k^-})\|_\phi.$$
Therefore,  we have  $\lim_{k \rightarrow \infty}\|v_k-f\|_\phi =0 $ and
$\lim_{k \rightarrow \infty} \|C_{n_k}v_k-h\|_\phi =0 $ which gives that
$\lim_{k \rightarrow \infty} v_k=f$ and $\lim_{k \rightarrow \infty} C_{n_k}v_{n_k}=h.$
 So, $C_{n_k}(U) \cap V \neq \emptyset$ for some $k.$ Hence, the sequence  $(C_n)_{n \in \mathbb{N}_0}$
  is topological transitive.
   \end{proof}
The following corollary gives a characterization of topological
mixing property of  cosine operators on Orlicz space $L^\phi(G).$
Since the proof is similar to above theorem therefore we will omit
the proof.
 \begin{cor} Let $g \in G$ be an aperiodic element of $G$ and let $\phi$
 be a $\Delta_2$-regular Young function. Let $w, \,w^{-1} \in L^\infty(G).$
  If $C_n:= \frac{1}{2}(T^n_{g,w}+S^n_{g,w})$ is cosine operator on $L^\phi(G)$
   then the following statements are equivalent.
    \begin{itemize}
        \item[(i)] $(C_n)_{n \in \mathbb{N}_0}$ is topological mixing.
        \item[(ii)] For each non-empty compact subset $K \subset G$ with
         $\lambda(K)>0,$ there exist sequences of Borel sets $(E_n)$ such that
        $$\lim_{n \rightarrow \infty} \sup_{\nu \in \Omega} \int_{K \backslash E_n} |\nu(x)|\,
        d\lambda(x)=0$$ and the two sequence
        $$\varphi_n= \prod_{j=1}^n w*\delta_{g^{-1}}^j\,\,\,\,\,\text{and}\,\,\,\,
        \tilde{\varphi_n}= \left( \prod_{j=1}^n w*\delta_g^j \right)^{-1}$$ satisfy
        $$ \lim_{n \rightarrow \infty} \sup_{\nu \in \Omega} \int_{E_n} \varphi_{n}(x)|\nu(xg^{n})|\,
         d\lambda(x)=0, $$
        $$ \lim_{n \rightarrow \infty} \sup_{\nu \in \Omega} \int_{E_n}
        \tilde{\varphi}_{n}(x)|\nu(xg^{n})|\, d\lambda(x)=0.$$
         \end{itemize}
    \end{cor}
We formulate the discrete version of Theorem \ref{T1}. If $G$ is
discrete group with the counting measure as its Haar measure. Then
the set $E_k$ is nothing but set $K$ itself. Therefore, we have the
following result.
\begin{cor} \label{2.3}
Let $g \in G$ be a non-torsion element of $G$ and let $\phi$
be a $\Delta_2$-regular Young function. Let $w, \,w^{-1} \in L^\infty(G).$
If $C_n:= \frac{1}{2}(T^n_{g,w}+S^n_{g,w})$ is cosine operator on $L^\phi(G)$
then the following statements are equivalent.
\begin{itemize}
    \item[(i)] $(C_n)_{n \in \mathbb{N}_0}$ is topological transitive.
    \item[(ii)] For each non-empty finite subset $K \subset G$, there exist
    sequences of Borel sets $(E_k^+)$ and $(E_k^-)$ in $K,$ and a sequence
    $(n_k)$ of positive numbers such that for $K=E_k^+ \cup E_k^-,$ we have the two sequence
    $$\varphi_n= \prod_{j=1}^n w*\delta_{g^{-1}}^j\,\,\,\,\,\text{and}\,\,\,\, \tilde{\varphi_n}=
    \left( \prod_{j=1}^n w*\delta_g^j \right)^{-1}$$ satisfy
    $$ \lim_{k \rightarrow \infty} \sup_{\nu \in \Omega} \sum_{x \in K}
    \varphi_{n_k}(x)|\nu(xg^{n_k})| = 0, \,\,\,\lim_{k \rightarrow \infty}
     \sup_{\nu \in \Omega} \sum_{x \in K}
    \tilde{\varphi}_{n_k}(x)|\nu(xg^{n_k})| =0,$$
    $$\lim_{k \rightarrow \infty} \sup_{\nu \in \Omega} \sum_{x \in E_k^+}
    \varphi_{2n_k}(x)|\nu(xg^{2n_k})|=0,\,\,\text{and}\,\, \lim_{k \rightarrow \infty}
     \sup_{\nu \in \Omega} \sum_{ x \in E_k^-}
    \tilde{\varphi}_{2n_k}(x)|\nu(xg^{2n_k})| =0.$$
\end{itemize}
\end{cor}
Here we present a characterization of topological trasitivity for a
finite sequence of weighted cosine operators. We set the following
notations for the sequence.
 For a fix $M \in \mathbb{N}.$ Let $\{g_l\}_{1 \leq l \leq M}$ and $\{w_l\}_{1 \leq l \leq M}$
  be the sequences of aperiodic elements of group $G$ and positive weight respectively.
   Then $\{T_{g_l, w_l}\}_{1 \leq l \leq M}$ is a sequence of weighted translation operators.
    We have the following characterization.
\begin{thm} \label{2.4} Let $\{g_l\}_l$ and $\{w_l\}_l$ be the sequences
 of aperiodic elements and positive weights respectively such that $w_l, \,w_l^{-1} \in L^\infty(G).$ Let
$C_{l,n}:=\frac{1}{2}(T^n_{g_l,w_l}+S^n_{g_l,w_l})$ be the cosine
operators on $L^\phi(G)$ for $1\leq l \leq M$, where $T_{g_l,w_l}$ is
the weighted translation operator.  Then the following statements
are equivalent.
    \begin{itemize}
        \item[(i)] $(C_{1,n}\oplus C_{2,n}\oplus \cdot\cdot\cdot \oplus C_{M,n})_{n \in \mathbb{N}_0}$
         is topologically transitive.
        \item[(ii)] For each non-empty compact subset $K \subset G$ with $\lambda(K)>0,$
there is some sequence $(n_k)$ of positive integers such that for
$1\leq l \leq M$,
        there exist sequences of Borel sets $(E_{l,k})$, $(E^+_{l,k})$ and $(E^{-}_{l,k})$ such
        that for $E_{l,k}=E^+_{l,k} \cup E^{-}_{l,k}$, we have
        $$\lim_{k \rightarrow \infty} \sup_{\nu \in \Omega} \int_{K \backslash E_{l,k}}
         |\nu(x)|\, d\lambda(x)=0$$
         and the two sequence
        $$\varphi_{l,n_k}= \prod_{j=1}^{n_k} w_l*\delta_{g_l^{-1}}^j\,\,\,\,\,\text{and}
        \,\,\,\, \tilde{\varphi}_{l,n_k}= \left( \prod_{j=1}^{n_k} w_l*\delta_{g_l}^j \right)^{-1}$$ satisfy
        $$ \lim_{k \rightarrow \infty} \sup_{\nu \in \Omega} \int_{E_{l, k}}
        \varphi_{l,n_k}(x)|\nu(xg^{n_k})|\, d\lambda(x)=0, $$
        $$ \lim_{k \rightarrow \infty} \sup_{\nu \in \Omega} \int_{E_{l,k}}
        \tilde{\varphi}_{l,n_k}(x)|\nu(xg^{n_k})|\, d\lambda(x)=0,$$
        $$ \lim_{k \rightarrow \infty} \sup_{\nu \in \Omega} \int_{E^+_{l,k}}
         \varphi_{l,2n_k}(x)|\nu(xg^{2n_k})|\, d\lambda(x)=0, $$
        $$ \lim_{k \rightarrow \infty} \sup_{\nu \in \Omega} \int_{E^{-}_{l, k}}
         \tilde{\varphi}_{l,2n_k}(x)|\nu(xg^{2n_k})|\, d\lambda(x)=0.$$
         \end{itemize}
    \end{thm}
    \begin{proof} (i)$\Rightarrow$(ii). Let $K$ be a compact subset of $G$ such that
    $\lambda(K)>0.$ Since $(C_{1,n}\oplus C_{2,n}\oplus \cdot\cdot\cdot
    \oplus C_{N,n})_{n \in \mathbb{M}_0}$ is topological transitive,
    for $\epsilon \in (0,1),$ there exist $f_l \in L^\phi(G)$ and $m \in N$ such that for
     $1\leq l \leq M,$ we have
        $$\| f_l- \chi_K\|_\phi < \epsilon^2\,\,\,\,\,
        \text{and}\,\,\,\, \|C_{l,m}f_l+\chi_K\|_\phi< \epsilon^2.$$

        Further, to complete the proof follow the proof of part (i) $\Rightarrow$ (ii)
        of Theorem \ref{T1} to get desired conditions on weights $w_l$ for each $l.$

        (ii) $\Rightarrow$ (i). Let $U_l$ and $V_l$ be non-empty open subsets of $L^\phi(G).$
         Since $\phi$ is $\Delta_2$-regular we can choose two non-zero
         functions $f_l$ and $h_l$ in $C_c(G)$ such that $f_l \in U_l$
          and $h \in V_l.$ Set $K=\spt(f_l) \cup \spt(h_l).$ Let $E_{l,k}
           \subset K$ and it satisfies condition (ii).
            Now, imitate the proof of (ii) $\Rightarrow$ (i) of Theorem \ref{T1}
            to get that $C_{l,n_k}(U_l) \cap V_l \neq \emptyset$ for each $l, \, 1\leq l \leq M.$
         \end{proof}

    \begin{exmp}
    Let $G=\mathbb{Z}$. Fix an aperiodic element $g\in \mathbb{Z}$
    with $g\geq 1$.
     Define the Young's function $\phi(x)=(1+|x|)\ln(1+|x|)-|x|$,
    and consider the weight function $$w(i)=\left\{
  \begin{array}{ll}
    \frac{1}{2}, & {i\geq 0}, \\
        \frac{3}{2}, & {i<0}.
          \end{array}
\right.$$ A direct computation needs to
    find the complementary of Young's function $\phi$, that is
    $\psi(x)=\exp({|x|})-|x|-1$.
 Note that $\phi$ is $\Delta_2$-regular and vanishes only
at zero.
 Choose an arbitrary $\nu \in
\Omega$. Then, $\sum_{n=-\infty}^{+\infty}\psi(|\nu(n)|)\leq 1$ if
and only if
$\sum_{n=-\infty}^{+\infty}(\exp(|\nu(n)|)-|\nu(n)|-1)\leq 1$. But
the last is established only if $|\nu|\leq 2$. Each compact subset
$K\subset \mathbb{Z}$ is a finite set, consisting of the integers
$a_1\leq a_2\leq \cdot\cdot\cdot \leq a_m$. Take
$E_i:=\{a_1,\cdot\cdot\cdot,a_i\}$ and for each $i\geq m$, define
$E_i:=E_m$. Put $n_k=i$, $E_k^+:=E_i$ and $E_k^-:=\emptyset$ in the
statement (ii) of Theorem \ref{T1}. In this circumstances, we have
$$\lim_{n\rightarrow \infty}\sup_{\nu \in \Omega}\sum_{j\in K\setminus E_n}|\nu(j)|=0.$$
In addition,  one may find $n_0\in \mathbb{N}$ such that
$a_1+n_0g\geq 0$. Hence, for each $n\geq n_0$ we have,
\begin{eqnarray*}
& \sup_{\nu \in \Omega}\sum_{i\in E_n}\varphi_n(i)|\nu(i+ng)| \leq
2\sum_{i\in E_n}\varphi_n(i)\\
&= 2\sum_{i\in
E_n}w(i+g)w(i+2g)\cdot\cdot\cdot w(i+ng)\\
&\leq 2\sum_{i\in E_n}w(a_1+g)w(a_1+2g)\cdot\cdot\cdot\\
& w(a_1+n_0
g)w(a_1+(n_0+1)g)\cdot\cdot\cdot w(a_1+n g)\\
&\leq 2\sum_{i\in
E_n}(\frac{1}{2})^{n-n_0}w(a_1)^{|a_1|}\\
&= 2card(E_n)(\frac{1}{2})^{n-n_0}w(a_1)^{|a_1|}\\
& \leq 2m(\frac{1}{2})^{n-n_0}w(a_1)^{|a_1|}\rightarrow 0,
\end{eqnarray*}
as $n\rightarrow \infty$. Here, $card(E_n)$ means the cardinality of
the set $E_n$.\\ Similarly, there exists $t_0\in \mathbb{N}$ such
that $a_m-t_0g\leq 0$ and so for each $n\geq t_0$,
\begin{eqnarray*}
& \sup_{\nu \in \Omega}\sum_{i\in E_n}\tilde{\varphi_n}(i)|\nu(i+n
g)| \leq
2\sum_{i\in E_n}\tilde{\varphi_n}(i)\\
&= 2\sum_{i\in
E_n}w^{-1}(i-g)w^{-1}(i-2g)\cdot\cdot\cdot w^{-1}(i-n g)\\
&\leq 2\sum_{i\in E_n}w^{-1}(a_m-g)w^{-1}(a_m-2g)\cdot\cdot\cdot\\
& w^{-1}(a_m-t_0 g)w^{-1}(a_m-(t_0+1)g)\cdot\cdot\cdot w^{-1}(a_m-n g)\\
&\leq 2\sum_{i\in
E_n}(\frac{2}{3})^{n-t_0}w^{-1}(a_m)^{|a_m|}\\
& \leq 2m(\frac{2}{3})^{n-t_0}w^{-1}(a_m)^{|a_m|}\rightarrow 0,
\end{eqnarray*}
as $n\rightarrow \infty$.  The other statements of Theorem \ref{T1}
can be verified in this way. Therefore, by Theorem \ref{T1}, the
corresponding sequence of  cosine operators to the weight $w$ and
$g$, is topological transitive.
    \end{exmp}

\section*{Acknowledgment}
Vishvesh Kumar thanks the Council of Scientific and Industrial Research, India, for its senior
 research fellowship.

\nocite{*}

\end{document}